\documentclass[a4paper,12pt]{amsart}
\usepackage{geometry}
\usepackage{euscript,eufrak,verbatim}
\usepackage[psamsfonts]{amssymb}
\usepackage[usenames]{color}
\usepackage{graphicx}
\usepackage[colorlinks,linkcolor=red,anchorcolor=blue,citecolor=blue]{hyperref}
\newtheorem{theorem}{Theorem}[section]
\newtheorem{lemma}[theorem]{Lemma}

\newtheorem{proposition}[theorem]{Proposition}
\newtheorem{corollary}[theorem]{Corollary}

\def\Z{\mathbb{Z}}
\def\Q{\mathbb{Q}}

\def\Zp{\mathbb{Z}_{p}}

\begin{document}

\title[Exponential Riesz bases in  non-Archimedean l. c. a. groups ]{Exponential Riesz bases in non-Archimedean locally compact  Abelian groups}

\date{} 

\author[A. H. FAN]{Aihua Fan}
\address[A. H. FAN]{
1. Wuhan Institute for Math \& AI, Wuhan University, Wuhan 430072, PR China\\
2. LAMFA, UMR 7352 CNRS, University of Picardie, 33 rue Saint Leu, 80039 Amiens, France
}
\email{ai-hua.fan@u-picardie.fr}

\author[S. L. FAN]{Shilei Fan}
\address[S. L. FAN]{
1. School of Mathematics and Statistics, and Hubei Key Lab--Math. Sci., Central China Normal University, Wuhan 430079, China\\
2. Key Laboratory of Nonlinear Analysis \& Applications (Ministry of Education), Central China Normal University, Wuhan 430079, China
}
\email{slfan@ccnu.edu.cn}

\thanks{A.H. FAN  was supported by NSF of China (Grant No. 12231013); S. L. FAN was partially supported by NSFC (grants No. 12331004) and NSF of Xinjiang Uygur Autonomous Region (Grant No. 2024D01A160).}

\begin{abstract}
This paper establishes  two fundamental results on the existence of exponential Riesz basis  
in non-Archimedean locally compact Abelian groups: the existence of Riesz basis of exponentials for all finite unions of balls and the non-existence of such basis  for some  bounded sets. 
\end{abstract}
\maketitle

 \section{Introduction}
 Let \( \mathcal{H} \) be a separable Hilbert space. A sequence of elements \( \{ f_n \}_{n=1}^\infty \subset \mathcal{H} \) is called a \emph{Riesz basis} if 
 \( f_n = T e_n \)  (\( \forall n \geq 1 \))  for some bounded invertible operator \( T: \mathcal{H} \to \mathcal{H} \) and for some   orthonormal basis  \( \{ e_n \}_{n=1}^\infty \) of \( \mathcal{H} \). Equivalently (see, e.g., \cite[Theorem~3.6.6]{Chr2016}), the sequence \( \{ f_n \}_{n=1}^\infty \) is a Riesz basis if it is complete in \( \mathcal{H} \) and there exists a constant \( K > 0 \) such that
\begin{equation}\label{eq:Riesz}
    \frac{1}{K} \sum_{n=1}^\infty |a_n|^2 \leq \left\| \sum_{n=1}^\infty a_n f_n \right\|^2 \leq K \sum_{n=1}^\infty |a_n|^2
\end{equation}
for all square-summable sequences \( \{ a_n \} \subset \mathbb{C} \).

    
   Let \( G \) be a locally compact abelian (l.c.a.) group with dual group \( \widehat{G} \). Elements of \( \widehat{G} \),   called \emph{characters}, are continuous group homomorphisms from \( G \) to the unit circle \( \mathbb{T}
  :=\{z: |z|=1\} \subset \mathbb{C} \). For the classical Euclidean additive group \( G = \mathbb{R}^d \), the dual group is naturally identified with \( \mathbb{R}^d \) and the characters take the form of complex exponentials
\[
e_\lambda(x) = e^{2\pi i \lambda \cdot x}, \quad \lambda \in \mathbb{R}^d.
\]
Thus, in this case, characters are nothing but exponential functions. 
We shall also refer to characters of a general l.c.a.\ group \( G \) as \emph{exponentials}. 

 Let \( S \subset G \) be a Borel set of positive and finite Haar measure of a l.c.a. group $G$. 
Consider the Hilbert space \( \mathcal{H} = L^2(S) \).  People are interested in the existence of structured bases in \( \mathcal{H} \) formed by exponentials. 

A natural question is  whether \( L^2(S) \) admits an orthogonal basis of exponentials. This question has been extensively studied in the Euclidean setting \( G = \mathbb{R}^d \), going back to the seminal work of Fuglede~\cite{Fuglede1974}. It is now rather well understood that the existence of exponential orthogonal basis  is related to the geometric structure of $S$. 
For example, Lev and Matolcsi ~\cite{LevMat2022} showed that a convex domain in \( \mathbb{R}^d \) has an orthogonal basis of exponentials if and only if it is a convex polytope that tiles \( \mathbb{R}^d \) by translations. However, classical shapes such as the disk or triangle in the plane fail to support such orthogonal bases; this was already noted in~\cite{Fuglede1974}.

A more flexible 
alternative is to seek a \emph{Riesz basis} of exponentials in \( L^2(S) \). This question is also delicate and remains open even for simple domains like the ball or  triangle in \( \mathbb{R}^2 \); see~\cite{KNO2023}. In the context of general l.c.a.\ groups, the problem becomes even more intricate, due to the lack of geometric intuition and the variety of topological and algebraic features that such groups may possess.

The study of exponential Riesz bases for finite unions of intervals in \(\mathbb{R}\) originated from practical applications in sampling theory for band-limited signals. Early progress was made through special cases: Kohlenberg \cite{Koh1953} solved the problem for two equal-length intervals, Bezuglaya and Katsnelson \cite{BK1993} considered intervals with integer endpoints, and Seip \cite{Seip1995} treated the general case of two intervals and certain subcases involving three or more intervals. A complete solution for arbitrary finite unions of intervals in \(\mathbb{R}\) was eventually obtained by Kozma and Nitzan \cite{KN2015}, who established the existence of exponential Riesz bases for such sets. More recently, Kozma, Nitzan, and Olevskii \cite{KNO2023}  proved that there exists a bounded measurable set \( S \subset \mathbb{R} \) of positive and finite measure  excluding Riesz bases of exponentials.

Kozma and Nitzan~\cite{KN2016} extended their earlier results to higher dimensions by proving that any finite union of rectangles in \( \mathbb{R}^d \), with edges parallel with the coordinate axes, admits a Riesz basis of exponentials. Grepstad and Lev~\cite{GL2014} showed that a bounded Borel set \( \Omega \subset \mathbb{R}^d \) admits a Riesz basis of exponentials if it \( k \)-tiles \( \mathbb{R}^d \) by translations along a full-rank lattice \( \Lambda \). Subsequently, Kolountzakis~\cite{Kol2015} provided a simpler and slightly more general proof of this result. Building upon this framework, Agora, Antezana and Cabrelli~\cite{AAC2015} further extended the theory to the setting of general locally compact abelian (l.c.a.) groups. More recently, Debernardi and Lev~\cite{DL2022} proved that any convex and centrally symmetric polytope in \( \mathbb{R}^d \), with centrally symmetric faces, admits a Riesz basis of exponentials without any additional arithmetic assumptions.

These results reveal the delicate interplay between geometric structure and spectral properties in Euclidean spaces. In this note, we aim to investigate analogous phenomena in the non-Archimedean setting.

 Fan et al   \cite{FFS2016, FFLS2019} proved  that a Borel set of positive and finite Haar measure in the field $\Q_p$ admit an orthogonal  basis  of  characters if and only if it tiles $\Q_p$ by translations. Moreover, it  is proved that such set are compact open  and  have $p$-homogeneous structure.
 It is important to note that in non-Archimedean settings such as \( \mathbb{Q}_p \), the classical tiling framework based on translating a set along a discrete subgroup (as in the Euclidean lattice case) does not apply. That is, there does not exist a discrete subgroup in \( \mathbb{Q}_p \) such that the translates of a set along this subgroup tile the whole space (up to measure zero).
 Indeed, \( \mathbb{Q}_p \) admits no discrete cocompact subgroup, which stands in sharp contrast to the Euclidean case \( \mathbb{R}^d \). This fundamental structural difference precludes the direct application of lattice tiling techniques that are central to the construction of exponential Riesz bases in the Euclidean setting.  We remark that although Agora, Antezana, and Cabrelli~\cite{AAC2015} extended the tiling-based Riesz basis theory to locally compact abelian (l.c.a.) groups, their approach crucially relies on the existence of a lattice in the group.
To overcome this obstacle, one must resort to alternative frameworks. These often involve compact open subgroups (such as \( \mathbb{Z}_p \subset \mathbb{Q}_p \)), coset decompositions, and tools from harmonic analysis on l.c.a. non-Archimedean groups. 

Recall that a topological group is  \emph{Polish} if it is separable and completely metrizable. A Polish group is said to be \emph{non-Archimedean} if it admits a compatible left-invariant ultrametric. It is known that a Polish group is non-Archimedean if and only if it has a neighborhood basis at the identity consisting of open subgroups (cf.~\cite[Theorem~2.4.1]{Gao2019}). Since 
all groups that we will consider are Polish,  we will drop the adjective  Polish. 
Also recall that a topological Abelian group \( G \) is called a \emph{locally compact abelian group}  (l.c.a. group for short) if its topology is Hausdorff and every point has a compact neighborhood. 

Our first result establishes the existence of exponential Riesz bases for the space of square-integrable functions  over a finite union of balls in any non-Archimedean l.c.a  group.

 \begin{theorem}\label{thm-main1}
 Let $G$ be a non-Archimedean  l.c.a.    
 group.  For any compact open subset $\Omega \subset G$, there exist  a set  $\Lambda\subset  \widehat{G}$ of  characters which forms a Riesz basis of $L^2(\Omega)$.
 \end{theorem}
 
Applied to the field $\Q_p$ of \( p \)-adic numbers, Theorem \ref{thm-main1} gives the following corollary.

\begin{corollary}
Let \( \Omega \subseteq \mathbb{Q}_p^d \) be a compact open subset ($d\ge 1$). Then $L^2(\Omega)$ admits  a Riesz basis consisting of characters in  $\widehat{\Q_p^d}$.
\end{corollary}

Recall that a Borel set \( \Omega \subset G \) is said to \( k \)-tile the group \( G \) with a translation set \( T \subset G \) if almost every point in \( G \) is covered exactly \( k \) times by the translated sets \( \Omega + t \) for \( t \in T \), that is,
\[
\sum_{t \in T} 1_{\Omega}(x - t) = k, \quad \text{for almost every } x \in G.
\]
As a direct consequence of \cite[Theorem~1.1]{F2024}, we obtain the following corollary.

\begin{corollary}
Let \( \Omega \subseteq \mathbb{Q}_p \) be a Borel set of positive and finite Haar measure. If \( \Omega \) \( k \)-tiles \( \mathbb{Q}_p \) by translations, then \( L^2(\Omega) \) admits a Riesz basis consisting of characters in \( \widehat{\mathbb{Q}_p} \).
\end{corollary}

We note that while the above result provides a positive answer in the one-dimensional space $\mathbb{Q}_p$, the situation in higher-dimensional spaces  \( \mathbb{Q}_p^d \) remains unclear. In particular, it is not known whether multi-tiling sets in \( \mathbb{Q}_p^d \) necessarily admit exponential Riesz bases. This presents an interesting direction for further research, especially in light of the deep structural differences between \( \mathbb{R}^d \) and \( \mathbb{Q}_p^d \).

%
%

In order to state our non-existence result, we need the following notion. Let \( S \) be a Borel subset  of a locally compact Abelian group $G$.
 A \emph{translation set} \( T \) (\(\subset G \)) of \( S \) is a set such that  translates \( \{S + t\}_{t \in T}\) are pairwise disjoint (i.e. the symmetric difference of two translates has zero Haar measure).  We always assume \( 0 \in T \). 
We are interested in translations of $S$ controlled within compact open subgroups. 
Thus, for a compact open subgroup $B$,  we
define the \emph{\( B \)-translation number} of \( S \) by
\[
\mathcal{N}_B(S) := \max\left\{ \# T : 0 \in T \subset B,\ T \text{ is a translation set of } S \right\}.
\]

\begin{proposition}\label{thm:criterion}
Let \( G \) be a non-Archimedean, non-discrete  l.c.a. group with Haar measure \( \mu \).
Suppose that \( \Omega \subset G \) is a bounded open set of the form
\[
\Omega = \bigsqcup_{n=1}^\infty (x_n + B_{n}),
\]
where \( x_n \in G \) and \( B_{n} \) are compact open subgroups of  \( G \), satisfying
\begin{itemize}
    \item[(1)] \(B_{ n}  \supset  B_{n+1} \) for all \( n \geq 1 \);
    \item[(2)] \( \displaystyle\limsup_{n \to \infty} \mathcal{N}_{B_{n}}(\Omega_n) = +\infty \), where \( \Omega_n := \bigsqcup_{k=n+1}^\infty (x_k + B_{k}) \).
\end{itemize}
Then there does not exist any subset \( \Lambda \subset \widehat{G} \) which forms a Riesz basis for \( L^2(\Omega) \).
\end{proposition}

To prove Proposition~\ref{thm:criterion}, we adopt the 
`translation' technique from ~\cite{KNO2023}. Thanks to the non-Archimedean structure of the group, the argument becomes significantly simpler than in the Euclidean case $\mathbb{R}$.
   In fact, such sets $\Omega$ in Proposition \ref{thm:criterion} can be constructed in any non-Archimedean, non-discrete l.c.a.\ group.
   So, we can confirm the following result.

\begin{theorem}
\label{thm-main2-general} 
Let $G$ be a non-Archimedean, non-discrete l.c.a. group with Haar measure $\mu$. 
There exists a bounded open set \( \Omega \subset G \) with \(0< \mu(\Omega) <+\infty \) such that no subset \( \Lambda \subset \widehat{G} \) forms a Riesz basis for \( L^2(\Omega) \). 
\end{theorem}

We now illustrate Theorem~\ref{thm-main2-general} with an explicit construction in the field $\Q_p$ of \( p \)-adic numbers. 
Denote by $\Z_p$ the ring of $p$-adic integers.

\begin{corollary}
Let \( \{m_n\} \) be a strictly increasing sequence of integers with \[ \limsup (m_{n+1} - m_{n}) = +\infty. \] Then the open set
\[
\bigcup_{n=1}^\infty   p^{m_n-1}+ p^{m_n}\Z_p
\]
excludes Riesz bases of continuous additive characters. \end{corollary}

For instance, the set
\(
\bigcup_{n=1}^\infty p^{n^2}+p^{n^2+1}\Z_p
\)
provides an explicit example lacking such basis.
\medskip

The paper is organized as follows. Section \ref{sect:2} is devoted to recall some basic notions about Riesz basis and $p$-adic numbers.
The $p$-adic setting provides a concrete and instructive framework in which key ideas will be clearly illustrated. 
Our main theorems will be first proved in the p-adic setting and subsequently we will explain how the arguments  extend to general setting of  non-Archimedean l.c.a  groups.
Theorem \ref{thm-main1} will be prove in Section  \ref{sect:3}, Proposition  \ref{thm:criterion}  and  Theorem \ref{thm-main2-general} will be proved  in Section \ref{sect:4}.

\section{preliminary}\label{sect:2}

In this section, we present an equivalent definition of Riesz basis which is more practical to use,  and some basic notions about non-Archimedean l.c.a. groups.


\subsection{Riesz sequences and frames in Hilbert spaces}
We first present an equivalent definition of Riesz basis. 
Let \( \mathcal{H} \) be a separable Hilbert space. A system \(\{ f_n \} \subset \mathcal{H}\) is called a \emph{Riesz sequence} if 
\begin{equation}\label{eq:Rieszb}
    \frac{1}{K} \sum_{n=1}^\infty |a_n|^2 \leq \left\| \sum_{n=1}^\infty a_n f_n \right\|^2 \leq K \sum_{n=1}^\infty |a_n|^2,  \quad \forall \{a_n\}\in \ell^2.
\end{equation}
These inequalities  are nothing but \eqref{eq:Riesz}. Here the completeness is not required (but is required by a Riesz basis).  
A system \(\{ f_n \} \subset \mathcal{H}\) is called a \emph{frame} if 
\begin{equation}\label{eq:Riesz2}
    \frac{1}{K} \|f \|^2 \leq \sum_{n=1}^\infty \left| \langle f,f_n \rangle \right|^2 \leq K \|f \|^2,  \quad  \forall  f \in \mathcal{H}.
\end{equation}

If \(\{f_n\}\) is a Riesz basis, there exists a \emph{dual system} \(\{g_n\}\) with \(\langle f_n, g_m \rangle = \delta_{n, m}\) and \(\|g_n\| \leq C\) for some constant \( C > 0 \). We can construct \(\{g_n\}\) as follows: let \(\{\phi_n\}\) be an orthonormal basis in \( \mathcal{H} \) and let \( T: \mathcal{H} \to \mathcal{H} \) be a bounded invertible operator such that \( T\phi_n = f_n \). Then the system defined by \( g_n := (T^*)^{-1}\phi_n \) satisfies the required properties.

As we will see below,
a duality argument establishes that \(\{ f_n \}\) is a Riesz basis if and only if it is minimal (i.e., no vector in the system lies within the closed span of the others) and satisfies the inequality \eqref{eq:Riesz2}, where \( K \) matches the constant in \eqref{eq:Rieszb}. This leads to the following characterization of Riesz bases, which  is well known in the literature (cf. \cite{KN2015}), for which we provide a proof  for the reader's convenience.

\begin{lemma} \label{lem-RSF}
  A system of vectors in a Hilbert space forms a Riesz basis if and only if it is both a Riesz sequence and a frame.
\end{lemma}

\begin{proof}
Suppose first that \(\{f_n\}\) is a Riesz basis. By definition, there exists a bounded invertible operator \(T: \mathcal{H} \to \mathcal{H}\) such that \(f_n = T \phi_n\) for an orthonormal basis \(\{\phi_n\}\) of \( \mathcal{H}\). 
Then, for any sequence \(\{a_n\} \in \ell^2\), we have
\[
\left\| \sum_n a_n f_n \right\| = \left\| T\left( \sum_n a_n \phi_n \right) \right\| \approx \left\| \sum_n a_n \phi_n \right\| = \left( \sum_n |a_n|^2 \right)^{1/2},
\]
where``\(\approx\)" means that the two terms bound each other with multiplicative constants depending on \(\|T\|\) and \(\|T^{-1}\|\). This shows that \(\{f_n\}\) is a Riesz sequence.
On the other hand, for any \(f \in \mathcal{H}\), we have
\[
\langle f, f_n \rangle = \langle f, T \phi_n \rangle = \langle T^* f, \phi_n \rangle,
\]
where \(T^*\) is the adjoint of \(T\). Therefore,
\[
\sum_n |\langle f, f_n \rangle|^2 = \sum_n |\langle T^* f, \phi_n \rangle|^2 = \|T^* f\|^2.
\]
Since \(T^*\) is bounded and invertible, there exists a constant \(K > 0\) such that
\[
\frac{1}{K} \|f\|^2 \leq \|T^* f\|^2 \leq  K \|f\|^2,
\]
for all \(f \in \mathcal{H}\). Hence, \(\{f_n\}\)  is a  frame.

Conversely, suppose that \(\{f_n\}\) is both a Riesz sequence and a frame. Define the synthesis operator \(F: \ell^2 \to \mathcal{H}\) by
\[
F(\{a_n\}) = \sum a_n f_n.
\]
The Riesz sequence property \eqref{eq:Rieszb}  implies that \(F\) is 
is injective and 
bounded, 
and hence has closed range.  It follows that \(F\) is a bounded bijection  from $\ell^2$ onto its range. Since \(\{f_n\}\) is a frame, the span of \(\{f_n\}\) is dense in \(H\), implying that the range of \(F\) is the whole space \(H\). Therefore, \(F\) is a bounded invertible operator from \(\ell^2\) onto \(H\).
Thus, \(\{f_n\}\) is the image of the canonical orthonormal basis \((\delta_n)\) of \(\ell^2\) under the bounded invertible operator \(F\), showing that \(\{f_n\}\) forms a Riesz basis.
\end{proof}

\subsection{Non-Archimedean locally compact Abelian groups}

Now we consider non-Archimedean l.c.a. groups. The non-Archimedean property results in a fundamentally different topological structure compared to that of real or complex Lie groups.

Here are several fundamental examples of non-Archimedean l.c.a groups:

\begin{enumerate}
    \item \textbf{The $p$-adic number field $\mathbb{Q}_p$}: The additive group $(\mathbb{Q}_p, +)$ is locally compact and totally disconnected. Its topology is defined via the $p$-adic norm, which is non-Archimedean.
    
    \item \textbf{The $p$-adic integers $\mathbb{Z}_p$}: This is the closed unit ball in $\mathbb{Q}_p$ under the $p$-adic norm, and forms a compact open subgroup of $\mathbb{Q}_p$. It is compact, totally disconnected, and a standard example of a non-Archimedean compact group.
    
    \item \textbf{Finite products of non-Archimedean groups}: For instance, $\mathbb{Z}_p^n$ is a locally compact abelian group with the product topology, and it inherits the total disconnectness, non-Archimedean structure from $\mathbb{Z}_p$.

    \item \textbf{Profinite groups as inverse limits}: Any profinite abelian group, such as $\varprojlim \mathbb{Z}/p^n\mathbb{Z}$, is a compact, totally disconnected group. These arise naturally as inverse limits of finite discrete groups and carry a canonical non-Archimedean topology.

    \item \textbf{Finite adeles} $\mathbb{A}_f$: The restricted direct product of $\mathbb{Q}_p$ over all primes $p$ with respect to $\mathbb{Z}_p$ is an important example arising in number theory, particularly in class field theory and automorphic representations (\cite{CF1967}).
\end{enumerate}

\subsection{The field of \( p \)-adic numbers}
Now we present 
the field \( \mathbb{Q}_p \) of $p$-adic numbers, viewed as an additive group, as a typical example of non-Archimedean l. c. a. groups.  Actually we will mainly work with the group \( \mathbb{Q}_p \), because similar arguments work for general
non-Archimedean l.c.a. groups.
 
Let us provide a brief overview of \( p \)-adic numbers.
 Let \( \mathbb{Q} \) denote the rational numbers and \( p \geq 2 \) a prime. Every nonzero rational number \( r \in \mathbb{Q} \) can be expressed as \( r = p^v \frac{a}{b} \), where \( v, a, b \in \mathbb{Z} \), \( \gcd(p, a) = 1 \) and \( \gcd(p, b) = 1 \). The \( p \)-adic absolute value of $r$ is defined as \( |r|_p = p^{-v_p(r)} \) for \( r \neq 0 \), with \( |0|_p = 0 \). The function $r\mapsto |r|_p$ satisfies:
\begin{itemize}
    \item[(i)] Non-negativity: \( |r|_p \geq 0 \), with equality only when \( r = 0 \);
    \item[(ii)] Multiplicativity: \( |rs|_p = |r|_p |s|_p \);
    \item[(iii)] Non-Archimedean property: \( |r + s|_p \leq \max\{ |r|_p, |s|_p \} \).
\end{itemize}
The field \( \mathbb{Q}_p \) of \( p \)-adic numbers is by definition the completion of \( \mathbb{Q} \) under \( |\cdot|_p \). The ring \( \mathbb{Z}_p \) of \( p \)-adic integers consists of elements 
of $\mathbb{Q}_p$ with \( |x|_p \leq 1 \). Any $p$-adic number \( x \in \mathbb{Q}_p \) admits its expansion
\begin{equation*}
    x = \sum_{n = v}^\infty x_n p^n \quad (v \in \mathbb{Z},\ x_n \in \{0,1,\ldots, p-1\},\ 
      x_v \neq 0),
\end{equation*}
where \( v_p(x) := v \) is called the \( p \)-valuation of \( x \). We denote  \( \{x\} = \sum_{n=v}^{-1} x_n p^n \), called the fractional part of $x$ and \( [x] = \sum_{n=0}^{\infty} x_n p^n \),  called the integral part of $x$.

The following function 
\[
\chi(x) = e^{2\pi i \{x\}}
\]
defines  non-trivial continuous additive character of the additive group \( \mathbb{Q}_p \) ( recall: \( \{x\} \) representing the fractional part of $x$). Notably, \( \chi \) is trivial on \( \mathbb{Z}_p \) but non-constant on 
the subgroup \( p^{-1}\mathbb{Z}_p \). The detailed definition of $\chi$, as a locally constant function,  is as follows
\begin{align}\label{one-in-unit-ball}
\chi(x) = e^{2\pi i k/p^n} \quad \text{for } x \in \frac{k}{p^n} + \mathbb{Z}_p \ \ (k, n \in \mathbb{Z}),
\end{align}
and $\chi$ shares the following property
\begin{align}\label{integral-chi}
\int_{p^{-n}\mathbb{Z}_p} \chi(x) \,dx = 0 \quad \text{for all } n\  {\color{blue} \geq 1}.
\end{align}

Every continuous character \( \chi_\xi \) of \( \mathbb{Q}_p \) can be obtained from $\chi$ via \( \chi_\xi(x) = \chi(\xi x) \). The correspondence \( \xi \mapsto \chi_{\xi} \) induces an isomorphism \( \widehat{\mathbb{Q}}_p \simeq \mathbb{Q}_p \), allowing identification of \( \xi \in \mathbb{Q}_p \) with \( \chi_\xi \in \widehat{\mathbb{Q}}_p \). For further details about the group  \( \mathbb{Q}_p \) and its dual group  \( \widehat{\mathbb{Q}}_p$, we refer to \cite{VVZ94}. For any subset $\Lambda\subset \Q_p$, we denote by  $E_\Lambda=\{\chi_{\lambda}: \lambda\in \Lambda\}$  the  corresponding system of characters or exponentials.

\medskip

\noindent Notation:
\\ \indent
$B(0, p^{n}): = p^{-n} \Zp$.  It is the (closed) ball centered at $0$ of radius $p^n$.

$B(x, p^{n}): = x + B(0, p^{n})$. 



$\mathbf{1}_A:$ the indicator function of a set $A$.

$\mathbb{L}: 
=\{\{x\}: x\in \Q_p\}$, a complete set of representatives of the quotient group $\Q_p/\Z_p$.

$\mathbb{L}_n:=p^{-n}\mathbb{L}$, a complete set of representatives of the quotient group $\Q_p/p^{-n}\Z_p$.
\medskip

Let $\Omega \subset \mathbb{Q}_p$ be a bounded measurable set with $0<\mu(\Omega) <+\infty$, where $\mu$ is the Haar measure of $\mathbb{Q}_p$. Our aim is to study the existence or non-existence of Riesz basis 
of the form $E_\Lambda$
consisting of  characters for the space $L^2(\Omega,\mu)$.
The following affine stability tells us that, without loss of generality,  we can assume that $\Omega$ is contained in $\mathbb{Z}_p$.

\begin{proposition}[Affine Stability]\label{prop:affine}
Let $E(\Lambda) = \{\chi_\lambda\}_{\lambda \in \Lambda}$ be a Riesz basis for $L^2(\Omega)$. 
\begin{enumerate}
    \item Any translated set $\Omega+a$ $(a \in \mathbb{Q}_p)$ also  admits  $E(\Lambda)$  as  a  Riesz basis.
    
    \item  Any the dilated set $b\Omega$ $(b \in \mathbb{Q}_p^*)$   admits $E(b^{-1}\Lambda) = \{\chi_{b^{-1}\lambda}\}_{\lambda \in \Lambda}$  as a Riesz basis.
\end{enumerate}

\end{proposition}

\begin{proof} These are simple facts, easy to see.

(1)  This is just because the translation operator $T_a f(x) = f(x-a)$ is unitary on $L^2(\mathbb{Q}_p)$, mapping $L^2(\Omega)$ onto $L^2(\Omega+a)$, and for $\chi_\lambda \in E(\Lambda)$ we have
\[
T_a\chi_\lambda(x) = \chi_\lambda(x-a) = \chi_\lambda(x)\overline{\chi_\lambda(a)}.
\]

(2) The dilation operator \(D_b f(x) := |b|_p^{-1/2} f(b^{-1}x)\) defines a unitary map from \(L^2(\Omega)\) onto \(L^2(b\Omega)\). Under this map, each exponential function $\chi_\lambda$ is  transformed to 
\[
D_b \chi_\lambda(x) = |b|_p^{-1/2} \chi_{b^{-1}\lambda}(x),
\]
so the image of \(E(\Lambda)\) under \(D_b\) is (up to a constant multiple) the exponential system \(E(b^{-1}\Lambda)\). Therefore, \(E(b^{-1}\Lambda)\) forms a Riesz basis for \(L^2(b\Omega)\).
\end{proof}

\section{Proof of Theorem \ref{thm-main1} }\label{sect:3}

One of ingredients of our proof of Theorem \ref{thm-main1} is the existence of Riesz basis for finite set of a discrete group. 
We start with a proof of this fact.

\subsection{Riesz basis for finite subsets of a discrete Abelian group}
The following lemma shows  that  every finite subset of a discrete abelian group admits a Riesz basis of exponential functions.
\begin{lemma}\label{lem-discrete}
Let \( G \) be a discrete abelian group, and let \( C \subset G \) be a finite subset with cardinality \( |C| = n \). Then there exists a subset \( D \subset \widehat{G} \) with \( |D| = n \) 
which forms a Riesz basis for \( L^2(C) \).
\end{lemma}

\begin{proof}
The Hilbert space \( L^2(C) \) consists of all complex-valued functions on $C$, and is naturally isomorphic to \( \mathbb{C}^n \) ($n$ being the cardinality of $C$), equipped with the standard inner product
\[
\langle f, g \rangle = \sum_{c \in C} f(c) \overline{g(c)}.
\]
Each character \( \chi \in \widehat{G} \)  restricted  to   \( C \)  gives up a vector \( \chi|_C \in \mathbb{C}^n \), namely an element of $L^2(C)$. Consider the evaluation map:
\[
\Phi : \widehat{G} \to \mathbb{C}^n, \quad \chi \mapsto \left( \chi(c_1), \dots, \chi(c_n) \right),
\]
where \(c_1, \dots, c_n \) are elements of $C$. Since \( \widehat{G} \) separates points in \( G \), the image space of \( \Phi \) is of dimension \( n \). Thus, there exist  \( n \) characters \( \chi_{d_1}, \dots, \chi_{d_n} \in \widehat{G} \) such that their restrictions \( \chi_{d_j}|_C \in \mathbb{C}^n \) are linearly independent.
Let \( D = \{\chi_{d_1}, \dots, \chi_{d_n}\} \).
 The system \(D \) is a basis of \( \mathbb{C}^n \cong L^2(C) \). But, in finite-dimensional Hilbert spaces, every basis is automatically a Riesz basis. This completes the proof.
\end{proof}

\subsection{Existence of  exponential Riesz bases  for compact open set in $\Q_p$}
Let \(\Omega \subset \mathbb{Q}_p\) be a compact open set. Through scaling and translation, we may assume \(\Omega \subset \mathbb{Z}_p\) (cf. Proposition~\ref{prop:affine}). Then the set \(\Omega\) can be written as a disjoint union of balls of equal radius:
\begin{align}\label{eq:compactopen}
\Omega = \bigsqcup_{c \in C} \left( c + p^\gamma \mathbb{Z}_p \right),
\end{align}
where \(\gamma \geq 0\) is an integer, and \(C \subset \{0, 1, \ldots, p^\gamma - 1\}\). Consider the set \(C\) as a subset of the finite cyclic group \(\mathbb{Z}/p^\gamma \mathbb{Z}\). 

The Pontryagin dual group of \(\mathbb{Z}/p^\gamma \mathbb{Z}\) is given by
\[
\Gamma_{\textcolor{red}\gamma} := \left\{ \frac{k}{p^\gamma} \in \mathbb{Q}/\mathbb{Z} \,:\, 0 \leq k \leq p^\gamma - 1 \right\},
\]
with the group operation defined as addition modulo 1. Namely, $(\mathbb{Z}/p^\gamma \mathbb{Z})\ {\widehat{}} =\Gamma_\gamma $. Indeed, each \(\xi \in \Gamma{\textcolor{red}\gamma}\) corresponds to an additive character \(\chi_\xi(x) = e^{2\pi i \xi x}\) of the group \(\mathbb{Z}/p^\gamma \mathbb{Z}\).
Notice that we have
$$
   \mathbb{Z}_p = \mathbb{Z} / p^\gamma \mathbb{Z} \oplus p^\gamma \mathbb{Z}_p.
$$
We can identity $ \Gamma_\gamma$ with  $p^{-\gamma}(\mathbb{Z} / p^\gamma \mathbb{Z})$ which is isomorphic to $p^{-\gamma}\mathbb{Z}_p /  \mathbb{Z}_p$.

By Lemma~\ref{lem-discrete}, for the subset \(C \subset \mathbb{Z}/p^\gamma \mathbb{Z}\), there exists a subset \(D \subset \Gamma{\textcolor{blue}\gamma}\) such that the exponential system \(E(D) = \{ \chi_d : d \in D \}\) forms a Riesz basis for \(L^2(C)\). Based on this fact and the fact that the group  \(p^\gamma \mathbb{Z}_p\)  has exponential orthonormal basis, we are going to prove the following Theorem \ref{thm:Qp-Riesz}.

 Recall that  
\[
\mathbb{L}_\gamma = \{0\} \cup \bigcup_{m=1}^{\infty} \left\{ \frac{k}{p^{\gamma + m}} : 1 \leq k \leq p^m,\; (k, p) = 1 \right\}.
\]
Consider  $D$ as a subset of $\Q_p$.
It follows that \( D + \mathbb{L}_\gamma \) is a direct sum in \( \mathbb{Q}_p \); that is, every element \( \lambda \in D + \mathbb{L}_\gamma \) admits a unique decomposition \( \lambda = \ell + d \) with \( \ell \in \mathbb{L}_\gamma \) and \( d \in D \). 

For such \( \lambda = \ell + d \), the associated character on \( \mathbb{Q}_p \) satisfies the identity
\begin{equation}\label{eq:CCC}
\chi_\lambda = \chi_{\ell + d} = \chi_\ell \cdot \chi_d,
\end{equation}
where $\chi_\ell$ and  $\chi_d$ are considered as characters on $\Q_p$.
Indeed, as  $\chi_d$ is locally constant on each coset of $p^\gamma \mathbb{Z}_p$,  it can also be viewed as a character on $\mathbb{Z}_p / p^\gamma \mathbb{Z}_p \cong \mathbb{Z} / p^\gamma \mathbb{Z}$.
If $x = a +b \in \mathbb{Q}_p$ with $a \in p^\gamma \mathbb{Z}_p$ and $b\in \mathbb{Q}_p/p^\gamma \mathbb{Z}_p$, the equality \eqref{eq:CCC} means
$$
    \chi_{\ell + d}(x) = \chi_\ell(x) \cdot \chi_d(b).
$$


\begin{theorem}[Exponential Riesz bases for compact open sets in \(\mathbb{Q}_p\)]\label{thm:Qp-Riesz}
Let \(\Omega \subset \mathbb{Z}_p\) be a compact open set with the decomposition
\begin{equation}\label{eq:compactopen}
\Omega = \bigsqcup_{c \in C} \left( c + p^\gamma \mathbb{Z}_p \right),
\end{equation}
where \(\gamma \geq 0\) is an integer and \(C \subset \mathbb{Z}/p^\gamma \mathbb{Z}\) is a finite set. Suppose \(D \subset {\textcolor{red}{\Gamma_\gamma}}\) is such that the exponential system \(E(D) := \{ \chi_d \}_{d \in D}\) forms a Riesz basis for \(L^2(C)\). Then the exponential system
\[
E(\Lambda) := \{ \chi_\lambda \}_{\lambda \in \Lambda}, \quad \Lambda := D + \mathbb{L}_\gamma,
\]
forms a Riesz basis for \(L^2(\Omega)\).
\end{theorem}

\begin{proof}
We employ Lemma \ref{lem-RSF} by verifying  \(E(\Lambda)\) is both a frame and Riesz sequence.

\medskip
\textbf{$E(\Lambda)$ is a Frame.}  
Since \(\Omega \subset \mathbb{Z}_p\), any function \(f \in L^2(\Omega)\)
is extended to a function $f1_\Omega$ in $L^2(\mathbb{Z}_p)$. As \(\Lambda \subset \widehat{\mathbb{Z}}_p =\mathbb{L}\), according to Parseval's identity we have 
\[
\|f\|^2_{L^2(\Omega)} 
 = \int_{\mathbb{Z}_p} |f \cdot \mathbf{1}_{\Omega}|^2 d\mu = \sum_{\lambda \in \mathbb{L}}|\widehat{f\cdot \mathbf{1}_\Omega}(\lambda)|^2 \geq \sum_{\lambda\in \Lambda}|\langle f, \chi_{\lambda} \rangle|^2,
\]
where $\langle \cdot, \cdot \rangle $ denotes the inner product  in $L^2(\Omega)$. 
To confirm the frame property, it remains to show
\begin{equation}\label{eq:remain}
\sum_{\lambda\in \Lambda}|\langle f, \chi_{\lambda} \rangle|^2 \geq \frac{1}{K} \|f\|^2_{L^2(\Omega)}
\end{equation}
for some constant $K>0$.
To this end, observe that 
 the system $E(\mathbb{L})$ forms an orthogonal basis for \(L^2(\mathbb{Z}_p)\).  Consequently, the system \(E(\mathbb{L}_{\gamma})\) forms an orthogonal basis 
for 
\(L^2(p^{\gamma}\mathbb{Z}_p)\) then for  \(L^2(c+p^{\gamma}\mathbb{Z}_p)\) for each \(c \in C\) (by a similar argument for proving Proposition \ref{prop:affine}). 
Thus, again by Parseval's identity, we have
$$
\int_{c+p^{\gamma} \mathbb{Z}_p} |f|^2 d\mu =  \sum_{\ell \in \mathbb{L}_{\gamma}} p^{\gamma} \left|\int_{c+p^{\gamma}\mathbb{Z}_p} f \cdot \chi_{-\ell} d\mu \right|^2.
$$
Summing over \(c \in C\) gives
\begin{equation}\label{eq:remain2}
\|f\|_{L^2(\Omega)}^2 = \sum_{c\in C} \int_{c+p^{\gamma} \mathbb{Z}_p} |f|^2 d\mu
%
= \sum_{\ell \in \mathbb{L}_{\gamma}} \sum_{c\in C} p^{\gamma}\left|\int_{c+p^{\gamma}\mathbb{Z}_p} f \cdot \chi_{-\ell} d\mu \right|^2.
\end{equation}
Now, we need to bound from above the last sum.  
Notice that we have 
 the decomposition \(\Lambda = D + \mathbb{L}_\gamma\) and for $\lambda =\ell +d \in \Lambda$ we have 
\[
\langle f, \chi_{\lambda} \rangle  = \langle f, \chi_{\ell + d} \rangle = \sum_{c\in C}  \int_{c + p^\gamma \mathbb{Z}_p} f \overline{\chi_\lambda}d\mu = \sum_{c \in C} \chi_{-d}(c) \int_{c + p^\gamma \mathbb{Z}_p} f \chi_{-\ell}\, d\mu,
\]
where we have used the relation \eqref{eq:CCC} and the fact that  \(\chi_d\) is constant on \(c + p^\gamma \mathbb{Z}_p\). 
Thus we get
\begin{equation}\label{eq:frame2}
\sum_{\lambda \in \Lambda} |\langle f, \chi_\lambda \rangle|^2 = \sum_{\ell \in \mathbb{L}_\gamma} \sum_{d \in D} \left| \sum_{c \in C} \chi_{-d}(c) \int_{c + p^\gamma \mathbb{Z}_p} f \chi_{-\ell}\, d\mu \right|^2.
\end{equation}
Since the character matrix 
\( (\chi_{-d}(c))_{c \in C, d \in D}\) is invertible, there exists a constant \(K' > 0\) such that
\[
\sum_{d \in D} \left| \sum_{c \in C} \chi_{-d}(c) a_c \right|^2 \geq \frac{1}{K'} \sum_{c \in C} |a_c|^2 \quad \text{for all } (a_c) \in \mathbb{C}^{|C|}.
\]
Apply this to $a_c = \int_{c + p^\gamma \mathbb{Z}_p} f \chi_{-\lambda}\, d\mu$. Then \eqref{eq:frame2} implies 
\begin{equation*} 
\sum_{\lambda \in \Lambda} |\langle f, \chi_\lambda \rangle|^2 \ge \frac{1}{K'} \sum_{\ell \in \mathbb{L}_\gamma} \sum_{c \in C} \left|  \int_{c + p^\gamma \mathbb{Z}_p} f \chi_{-\ell}\, d\mu \right|^2.
\end{equation*}
This, together with \eqref{eq:remain2},  implies 
the desired frame estimate \eqref{eq:remain} with $K=K' p^\gamma$.

\medskip
\textbf{$E(\Lambda)$ is a Riesz Sequence.}  
We now verify that 
for any sequence \( \{a_\lambda\} \in \ell^2(\Lambda) \),  we have
\[
\frac{1}{K} \sum_{\lambda\in \Lambda} |a_{\lambda}|^2 \leq \left\| \sum_{\lambda \in \Lambda} a_{\lambda}\chi_{\lambda} \right\|_{L^2(\Omega)}^2 \leq K \sum_{\lambda\in \Lambda} |a_{\lambda}|^2.
\]
The second inequality follows  from Parseval's identity based on the assumptions \( \Omega \subset \mathbb{Z}_p \) and \( \Lambda \subset \mathbb{L} \).

In order to prove the first inequality,  we expand  the squared norm 
    \begin{align*}
        \left\| \sum_{\lambda \in \Lambda} a_{\lambda}\chi_{\lambda} \right\|^2_{L^2(\Omega)} 
        = \sum_{\lambda,\lambda' \in \Lambda} a_\lambda \overline{a}_{\lambda'} \langle \chi_\lambda, \chi_{\lambda'} \rangle.
    \end{align*}
 Using the decomposition  \( \Lambda = \bigcup_{\ell \in \mathbb{L}_\gamma} (\ell + D)\), we further get 
    \begin{align}\label{eq:RR}
  \left\| \sum_{\lambda \in \Lambda} a_{\lambda}\chi_{\lambda} \right\|^2_{L^2(\Omega)}
        &= \sum_{\ell, \ell' \in \mathbb{L}_\gamma} \sum_{d,d' \in D} a_{\ell+d} \overline{a_{\ell'+d'}} \langle \chi_{\ell+d}, \chi_{\ell'+d'} \rangle.
    \end{align}
     But we have the orthogonality \[ \langle \chi_{\lambda}, \chi_{\lambda'} \rangle =  \int_\Omega  \chi_{\lambda-\lambda'} d\mu 
     =\sum_{c\in C} \int_{c +p^\gamma \mathbb{Z}_p}  \chi_{\lambda-\lambda'} d\mu = 0 \] 
     when \( |\lambda - \lambda'|_p \geq p^{-(\gamma+1)} \), according to \eqref{integral-chi}.
    In particular, for $\ell+d, \ell'+d'\in \Lambda$ with $\ell\neq \ell',$  we have  \[ |\ell+d-(\ell'+d')|_p=|\ell-\ell'|_p\geq p^{\gamma+1},\] and consequently 
    \( \langle \chi_{\ell+d}, \chi_{\ell'+d'}\rangle =0\).  
    
    Therefore we can simplify \eqref{eq:RR} to 
    \begin{align}\label{eq:RRR}
        \left\| \sum_{\lambda \in \Lambda} a_{\lambda}\chi_{\lambda} \right\|^2_{L^2(\Omega)} 
        &= \sum_{\ell \in \mathbb{L}_\gamma}  \left \|\sum_{d \in D} a_{\ell+d}\chi_{\ell+d}\right\|^2_{L^2(\Omega)}.
    \end{align}
       Fix  \( \ell \in \mathbb{L}_\gamma \).  
       As $\chi_{\ell+d}=\chi_\ell \chi_d$, we have
\begin{equation}\label{eq:R4}
\left\| \sum_{d \in D} a_{\ell+d}\chi_{\ell+d} \right\|^2_{L^2(\Omega)} 
= \left\| \chi_{\ell} \sum_{d \in D} a_{\ell+d}\chi_{d} \right\|^2_{L^2(\Omega)} 
= \sum_{c \in C} p^{-\gamma} \left| \sum_{d \in D} a_{\ell+d} \chi_d(c) \right|^2.
\end{equation}
For the second equality, we have used the fact that  \( \chi_d \) is constant on each coset \( c + p^\gamma\mathbb{Z}_p \).
    However, as the matrix 
 \( (\chi_{-d}(c))_{c \in C, d \in D} \) is invertible,  there exists a constant $K>0$ such that 
    \begin{equation}\label{eq:R5}
        \sum_{c \in C} p^{-\gamma} \left| \sum_{d \in D} a_{\ell+d} \chi_d(c) \right|^2 \geq \frac{1}{K} \sum_{d \in D} |a_{\ell+d}|^2.
    \end{equation}
    From \eqref{eq:RRR}, \eqref{eq:R4} and \eqref{eq:R5} we get
    \[
        \left\| \sum_{\lambda \in \Lambda} a_{\lambda}\chi_{\lambda} \right\|^2 \geq \frac{1}{K} \sum_{\ell \in \mathbb{L}_\gamma} \sum_{d \in D} |a_{\ell+d}|^2 = \frac{1}{K} \sum_{\lambda \in \Lambda} |a_\lambda|^2.
    \]
    This finishes the proof. 
\end{proof}

\subsection{ Existence of  exponential Riesz bases  for compact open set  in non-Archimedean l.c.a.  groups}

Let \(G\) be a non-Archimedean l.c.a.  group with  Haar measure \(\mu\). A key structural property of such groups is the existence of a base of neighborhoods of the identity consisting of compact open subgroups. This allows us to develop a multi-scale   analysis similar to that in \( \mathbb{Q}_p \), and to construct function systems with desirable spectral properties.

A compact open subsets  is of the form
\[
\Omega = \bigsqcup_{c \in C} (c + H),
\]
where \(H\) is a fixed compact open subgroup of \(G\), and \(C\subset G\) is a finite set of coset representatives of the discrete group $G/H$.

Let \(\widehat{G}\) denote the Pontryagin dual group of \(G\), consisting of all continuous homomorphisms \(\chi: G \to \mathbb{T}\). The annihilator of \(H\) in \(\widehat{G}\) defined  by
\[
H^\perp = \{ \chi \in \widehat{G} \,:\, \chi(h) = 1 \text{ for all } h \in H \}
\]
 is a compact open subgroup of \(\widehat{G}\). 
 The quotient \(\widehat{G}/H^\perp\) is a discrete group, and there exists a natural topological isomorphism
\[
\widehat{G/H} \cong H^\perp,
\]
as established in \cite{HR2}. This identification allows us to transfer the Riesz basis construction on \(L^2(C)\) to the exponential system on \(L^2(\Omega)\).

By Lemma~\ref{lem-discrete} and the identification \(\widehat{G/H} \cong H^\perp\), there exists a subset \(D \subset H^\perp\) of size \(|C|\) such that the exponential system
\[
E(D) := \{ \chi_d \}_{d \in D}
\]
forms a Riesz basis for \(L^2(C)\), where \(\chi_d\) is interpreted as a function constant on cosets \(c + H\) and indexed by \(d \in H^\perp\).

Let \(L \subset \widehat{G}\) be a complete set of coset representatives for the quotient \(\widehat{G}/H^\perp\). Define the frequency set
\[
\Lambda := D + L = \{ d + \ell : d \in D,\, \ell \in L \}.
\]
Then each \(\chi_{\lambda} \in E(\Lambda)\) with \(\lambda = d + \ell\) can be written as
\[
\chi_{d+\ell}(x) = \chi_d(x) \chi_\ell(x),
\]
where \(\chi_d\) is constant on cosets \(c + H\), and \(\chi_\ell\) modulates over distinct cosets. The set \(E(\Lambda)\) can thus be viewed as a union of translates of \(E(D)\) by  \(\ell \in L\).

Since characters in different cosets of \(H^\perp\) are orthogonal in \(L^2(\Omega)\), and since \(E(D)\) is a Riesz basis  of $L^2(C)$, the full system \(E(\Lambda)\) forms a Riesz basis  for  \(L^2(\Omega)\) by the same argument as in the \(p\)-adic case.

\section{Proof of  Proposition \ref{thm:criterion} and  Theorem \ref{thm-main2-general}}\label{sect:4}

\subsection{Nonexistence of exponential Riesz bases in $\Q_p$}
Without loss of generality, we can assume $\Omega \subset \mathbb{Z}_p$. 
It is clear that $\{p^m \Z_p\}_{m\geq 0}$  is  a neighborhood base of $0$. 
Any open set $\Omega\subset  \Z_p$ admitting the canonical decomposition
\begin{align}\label{eq:open}
\Omega = \bigsqcup_{n=1}^\infty x_n+ p^{m_n}\Z_p
\end{align}
with $x_n$'s and $m_n$'s satisfying  
\begin{itemize}
    \item $ x_n+ p^{m_n}\Z_p$  are disjoint,
      \item  \( m_n \geq m_{n+1} \),
    \item $x_n+ p^{m_n-1}\Z_p \not\subset \Omega \).
\end{itemize}
\begin{theorem}
Let  $\Omega$ be an open set in $\Z_p$ of the form \eqref{eq:open}. 
Suppose 
\begin{equation}\label{eq:hypo}
\limsup_{n\to +\infty}  \mathcal{N}_{p^{m_n}\Z_p}(\Omega_{n})=+\infty.
\end{equation}
where \(\Omega_n=\bigcup_{k=1}^{\infty}x_{n+k}+p^{m_{n+k}}\Z_p\). 
Then $L^{2}(\Omega)$ does not admit  Riesz basis   composed of continuous additive characters.
\end{theorem}
\begin{proof}
We are going to show the conclusion by contradiction.
Suppose that $E(\Lambda)$ is a Riesz basis for $L^2(\Omega)$ with constant $K$ in \eqref{eq:Riesz}. 
Recall that $\chi_\lambda(x) =e^{2\pi i\{\lambda x\}} $.
Observe that 
$$
|\lambda - \lambda'|_p \leq 1\Longrightarrow \forall  x \in \mathbb{Z}_p,  \chi_\lambda(x) = \chi_{\lambda'}(x).
$$
Thus, we may assume 
\[\Lambda \subset \widehat{\mathbb{Z}}_p=\mathbb{L} =\{\frac{k}{p^m}: m \ge 1, (k, p)=1\}.\]  

Let us consider the $L^2$-normalized  indicator functions 
\[
f_n(x) = \frac{1}{\sqrt{\mu(\Omega_n)}} \mathbf{1}_{\Omega_n}(x), \quad x \in \Omega.
\]
Since $E(\Lambda)$ is a Riesz basis of $L^2(\Omega)$, each function $f_n\in L^2(\Omega)$ can be expanded into  
\begin{align}\label{eq:intzp0}
f_n(x) = \sum_{\lambda \in \Lambda} c_\lambda \chi_\lambda(x)
\ \ \ 
{\rm with} \ \ \  \frac{1}{K} \leq \sum_{\lambda \in \Lambda} |c_\lambda|^2 \leq K.
\end{align} 
Since the sequence  $(c_\lambda)$ is square summable, the following orthogonal expansion in $L^2(\mathbb{Z}_p)$ well defines a function  $\widetilde{f} \in L^2(\mathbb{Z}_p):$ 
\begin{align}\label{eq:intzp}
\widetilde{f_n}(x) := \sum_{\lambda \in \Lambda} c_\lambda \chi_\lambda(x) \quad {\rm with} \ \ \ 
\|\widetilde{f_n}\|_{L^2(\mathbb{Z}_p)}^2 = \sum_{\lambda \in \Lambda} |c_\lambda|^2<+\infty.
\end{align}
The function $\widetilde{f_n}$  is an extension of $f_n$ from $\Omega$ onto $\mathbb{Z}_p$.

Fix $t \in \mathbb{Z}_p$ and consider the translation $x \mapsto \widetilde{f_n}(x + t)$, which is still a function in $L^2(\mathbb{Z}_p)$. From \eqref{eq:intzp} we get 
\[
\widetilde{f}(x+t) = \sum_{\lambda \in \Lambda} c_\lambda \chi_\lambda(t) \chi_\lambda(x).
\]
This series converges in $L^2(\mathbb{Z}_p)$, a fortiori in $L^2(\Omega)$.  So, restricted on $\Omega$, the last series is the expansion  of $\widetilde{f}(x+t) 1_\Omega(x)$ in the Riesz basis.
As $|\chi_\lambda(t)| = 1$, we have    
\[
\|\widetilde{f}(\cdot+t) \|_{L^2(\Omega)}^2 \geq \frac{1}{K} \sum_{\lambda \in \Lambda} |c_\lambda|^2 \geq \frac{1}{K^2},
\]
where the first inequality follows from the definition condition \eqref{eq:Riesz} and the second one from \eqref{eq:intzp0}.
Now making the change of variable $x\mapsto x+t$ leads to
\begin{equation}\label{eq:Oy}
\|\widetilde{f_n}\|_{L^2(\Omega + t)}^2 \geq \frac{1}{K^2}.
\end{equation}
If we assume  $t \in  p^{m_n}\Z_p$, we have
\[
\Omega + t = \left( \bigsqcup_{k=1}^n x_k+ p^{m_k}\Z_p \right) \sqcup (\Omega_n + t).
\]
As  the function $f_n$ vanishes on  $\bigcup_{k=1}^n x_k+ p^{m_k}\Z_p$, so does its extension $\widetilde{f}_n$. 
Then, from \eqref{eq:Oy}, we get
\[
\int_{\Omega_n + t} |\widetilde{f_n}(x)|^2 d\mu(x) \geq \frac{1}{K^2}, \quad \forall t \in  p^{m_{n}}\Z_p.
\]
Now take a $m_n$-translation set $T$ of $\Omega_n$ of maximal cardinality $\mathcal{N}_{p^{m_{n}}\Z_p}(\Omega_n)$. 
Then  
\[
\int_{\mathbb{Z}_p} |\widetilde{f_n}(x)|^2 d\mu(x) \geq \sum_{t\in T} \int_{\Omega_n +t}  |\widetilde{f_n}(x)|^2 d\mu(x)  \geq \frac{\mathcal{N}_{m_n}(\Omega_n)}{K^2}.
\]
This together with the assumption \eqref{eq:hypo} 
forces  
\[
\int_{\mathbb{Z}_p} |\widetilde{f_n}(x)|^2 d\mu(x) = +\infty,
\]
a contradiction to  \eqref{eq:intzp}.
\end{proof}

\subsection{Nonexistence of exponential Riesz bases in general non-Archimedean l.c.a. groups}
We now establish a general criterion that ensures the nonexistence of exponential Riesz bases in \( L^2(\Omega) \) for certain bounded open subsets of  $G$.

\begin{proof}[Proof of Proposition  \ref{thm:criterion} ]
 Fix a compact open subgroup \( U \subset G \) with normalized Haar measure \( \mu(U) = 1 \). 
Without loss of generality, we assume \( \Omega \subset U \), so that \( x_n \in U \) for all \( n \).

Suppose on the contrary that \( E(\Lambda) = \{\chi_\lambda\}_{\lambda \in \Lambda} \subset \widehat{G} \) is a Riesz basis for \( L^2(\Omega) \) with bound \( K \) as in~\eqref{eq:Riesz}. By the local constancy of characters on compact open subgroups, we may assume \( \Lambda \subset \bigcup_{n\ge 1} B_n^\perp \), where \( B_n^\perp := \{ \chi \in \widehat{G} : \chi|_{B_n} \equiv 1 \} \).

Define
\[
f_n(x) := \frac{1}{\sqrt{\mu(\Omega_n)}} \mathbf{1}_{\Omega_n}(x), \quad x \in \Omega.
\]
Then by assumption,
\begin{equation}\label{eq:Riesz-general}
f_n(x) = \sum_{\lambda \in \Lambda} c_\lambda \chi_\lambda(x), \qquad \frac{1}{K} \le \sum_{\lambda \in \Lambda} |c_\lambda|^2 \le K.
\end{equation}

Extend this function to all of \( U \) by defining
\[
\widetilde{f_n}(x) := \sum_{\lambda \in \Lambda} c_\lambda \chi_\lambda(x), \quad x \in U,
\]
which converges in \( L^2(U) \), since the coefficients are square summable.

Now for any \( t \in U \), by the same argument as in the $p$-adic setting, we obtain that 
\[
\|\widetilde{f_n}(\cdot + t)\|_{L^2(\Omega)}^2 \ge \frac{1}{K^2}.
\]
Now making the change of variable $x\mapsto x+t$ leads to

\[
\|\widetilde{f_n}\|_{L^2(\Omega + t)}^2 \ge \frac{1}{K^2}, \quad \forall t \in U.
\]
Observe that \( \widetilde{f_n} \equiv 0 \) on \( \bigcup_{k=1}^n x_k+ B_{k} \). Hence,
\[
\int_{\Omega_n + t} |\widetilde{f_n}(x)|^2 \, d\mu(x) \ge \frac{1}{K^2}, \quad \forall t \in U_{m_n}.
\]

Now let \( T \subset B_{n} \) be a maximal translation set of \( \Omega_n \), so that \( \#T = \mathcal{N}_{B_{n}}(\Omega_n) \). Then, we have
\[
\|\widetilde{f_n}\|_{L^2(U)}^2 \ge \sum_{t \in T} \int_{\Omega_n + t} |\widetilde{f_n}(x)|^2 \, d\mu(x) \ge \frac{\mathcal{N}_{B_{n}}(\Omega_n)}{K^2}.
\]
Thus,
\[
\limsup_{n \to \infty} \|\widetilde{f_n}\|_{L^2(U)}^2 =+ \infty,
\]
contradicting the convergence of the expansion of \( \widetilde{f_n} \in L^2(U) \). \end{proof}

We now construct an explicit example of a bounded open set in a general non-Archimedean, non-discrete l.c.a. group that admits no exponential Riesz basis. This concrete example illustrates the applicability of Proposition~\ref{thm:criterion}.

\begin{proof}[Proof of Theorem \ref{thm-main2-general}]
Since $G$ is non-discrete, there exists   a nested compact open subgroups \( \{B_n\}_{n \ge 0} \subset G \) such that  
\begin{align}\label{limsup}
\lim_{n\to +\infty}   \# (B_{n}/ B_{n+1})=+\infty.
\end{align}

We choose a sequence of point $\{x_n\}_{n\geq 1}$ in $G$ such that 
$x_n \in  B_{n-1}\setminus B_{n}$ and define an open set  \[\Omega=\bigcup_{n=1}^{\infty} x_n+B_n.\]
Let
\[\Omega_n=\bigcup_{k=n+1}^{\infty} x_k+B_k.\]  
Noting that $\Omega_{n} \subset B_n$ for each $n\geq1$. Hence, we obtain that \[ \Omega_{n}  \subset 
(x_{n+1}+ B_{n+1}) \cup B_{n+1}.\]
By  \eqref{limsup}, we have 
\begin{align*}
\lim_{n\to +\infty}   \mathcal{N}_{B_{n}}(\Omega_n)=+\infty. 
\end{align*}
By Proposition \ref{thm:criterion}, there does not exist any subset \( \Lambda \subset \widehat{G} \) which forms a Riesz basis for \( L^2(\Omega) \).
\end{proof}

\bibliographystyle{siam}


\end{document}